\documentclass{amsart}
\usepackage{amssymb,euscript,amsmath, mathrsfs, amscd}
\usepackage[pdftex]{graphicx}
\usepackage[pdftex]{color}

\newcounter{ENUM}

\newcommand{\margh}[1]{}

\input xy
\xyoption{all}
\CompileMatrices

\def\PP{{\mathbb P}}

\def\sL{{\mathscr L}}

\def\sO{{\mathscr O}}

\def\fg{{\mathfrak g}}

\def\Pic{\operatorname{Pic}}

\def\EH{\operatorname{EH}}

\newtheorem{thm}{Theorem}[section]
\newtheorem{prop}[thm]{Proposition}
\newtheorem{lem}[thm]{Lemma}
\newtheorem{cor}[thm]{Corollary}

\theoremstyle{definition}
\newtheorem{defn}[thm]{Definition}

\newtheorem{sit}[thm]{Situation}
\theoremstyle{remark}
\newtheorem{notn}[thm]{Notation}
\newtheorem{rem}[thm]{Remark}
\newtheorem{rems}[thm]{Remarks}

\numberwithin{equation}{section}
\numberwithin{figure}{section}

\begin{document}
\title{A simple characteristic-free proof of the Brill-Noether theorem}
\author{Brian Osserman}
\thanks{The author was partially supported by NSA grant H98230-11-1-0159
during the preparation of this work.}
\begin{abstract} We describe how the use of a different degeneration from 
that considered by Eisenbud and Harris leads to a simple and 
characteristic-independent proof of the Brill-Noether theorem using 
limit linear series. As suggested by the degeneration, we prove an
extended version of the theorem allowing for imposed ramification at up
to two points. Although experts in the field have long been aware of the
main ideas, we address some technical issues which arise in proving the
full version of theorem.
\end{abstract}
\maketitle

\section{Introduction}

The Brill-Noether theorem addresses the fundamental question of the
dimension and nonemptiness of spaces of linear series on smooth projective
curves. The original proof was characteristic independent, although 
Griffiths and Harris \cite{g-h1} wrote their portion of the argument
in the context of complex varieties. Gieseker's proof of
the Petri conjecture \cite{gi1}
was explicitly written as holding in all characteristics, and
gave a new proof of the Griffiths-Harris result.
The later degeneration proof
by Eisenbud and Harris using limit linear series
(Theorem 5.1 of \cite{e-h4}; see also \cite{e-h8}\footnote{The characteristic
$0$ hypothesis is used implicitly in Lemma 8.})
in fact used the characteristic-$0$ hypotheses in nontrivial
ways. This last proof is quite simple, so it is desirable to have
a characteristic-independent version. At the same time, in Theorem 4.5 of
\cite{e-h1}, Eisenbud and Harris generalized the Brill-Noether theorem to 
consider linear series with imposed ramification. Not only does their
argument use characteristic $0$, but as stated the result fails in positive
characteristic due to inseparability phenomena. 

The purpose of the present note is to give a simple proof of the most 
general form of the Eisenbud-Harris result which holds in all 
characteristics: the case that ramification is imposed at at most two 
points. As is known to the experts in the field, the dependence on 
characteristic in the earlier Eisenbud-Harris proof of the classical
Brill-Noether theorem can be largely avoided if one
changes the degeneration considered, working instead with a degeneration
to a chain of elliptic curves as considered by Welters in 
\cite{we4}. We describe this proof, and in the process observe that there
is a nontrivial technical obstacle to this approach which can be 
circumvented via the alternative construction of limit linear series
introduced in \cite{os8} (see Remark \ref{rem:technicality} for details). 
Our main result is the following:

\begin{thm}\label{thm:bn-plus-2} Let $C$ be a smooth, projective curve
of genus $g$ over an algebraically closed field $k$, 
and $P_1, P_2$ distinct points on $C$.  Given $r,d$, and
sequences $0 \leq \alpha^i_0 \leq \dots \leq \alpha^i_r \leq d-r$ for
$i=1,2$, set
$$\rho:=(r+1)(d-r)-rg-\sum_{i=1}^2 \sum_{j=0}^r \alpha^i_j.$$
Now, suppose that we have
\begin{equation}\label{eq:ram-ineq}
\sum_{j:\alpha^1_j+\alpha^2_{r-j}+r \geq d+1-g} \alpha^1_j+\alpha^2_j+r-d+g
\leq g.
\end{equation}
Then the space $G^r_d(C,P_1,P_2,\alpha^1,\alpha^2)$ 
parametrizing $\fg^r_d$s on $C$ with ramification sequence at least 
$\alpha^i$ at $P_i$ for $i=1,2$ is nonempty with every component of
dimension at least $\rho$.

Furthermore, if $(C,P_1,P_2)$ is a general $2$-marked curve, then in
fact $G^r_d(C,P_1,P_2,\alpha^1,\alpha^2)$ is nonempty if and only if
\eqref{eq:ram-ineq} is satisfied, and in this case is pure of dimension 
$\rho$.
\end{thm} 

Recall the ramification terminology: if $(\sL,V)$ is a $\fg^r_d$ on $C$,
and $P \in C$, the ramification sequence $\alpha_0(P),\dots,\alpha_r(P)$
is determined by $\alpha_j(P)=a_j(P)-j$ for all $j$, where
$a_0(P),\dots,a_r(P)$ is the strictly increasing sequence of orders of
vanishing at $P$ of sections in $V$.

\begin{rems} 
(i) One easily verifies that \eqref{eq:ram-ineq} implies that
$\rho \geq 0$.
In addition, if the $\alpha^i_j$ are all equal to $0$, so
that we are in the classical case without imposed ramification, and if
$\rho \geq 0$, then \eqref{eq:ram-ineq} is satisfied.
Thus, we recover in particular the classical statement of
the Brill-Noether theorem.

(ii) Note that the theorem fails in positive characteristic for ramification
imposed at more than $2$ points, even in the case of $\fg^1_d$'s on $\PP^1$,
so the statement is sharp.

(iii) Although typically the dimensional upper bound for general curves is
viewed as deeper than the nonemptiness statement, in the context of the type
of degeneration argument we use, the opposite is true: one may prove the 
dimension statement without the nonemptiness statement, but not vice versa.

(iv) The restriction to algebraically closed base fields is a matter of
convenience, insofar as the $G^r_d$ spaces may be defined over an arbitrary
base field (or scheme), and commute with base change. However, although one
obtains a more general statement via this reduction, the notion of a general 
curve is less useful in the arithmetic setting, since there is no guarantee 
of having even one $k$-rational point inside a Zariski open subset of the 
moduli space.
\end{rems}

\begin{rem}\label{rem:reduce-to-special} We recall a standard reduction
step for proving the Brill-Noether theorem.
The dimensional lower bound is immediate from the construction of
the space $G^r_d(C,P_1,P_2,\alpha^1,\alpha^2)$, which realizes the
space as an intersection of relative Schubert cycles inside a relative
Grassmannian over $\Pic^d(C)$. Furthermore, this
construction works for families of curves to produce proper $G^r_d$
moduli spaces whose components always have relative dimension at least
$\rho$; it follows that to prove Theorem \ref{thm:bn-plus-2}, it suffices
to produce a single smooth curve $C$ (over any extension of the base
field) for which $G^r_d(C,P_1,P_2,\alpha^1,\alpha^2)$ is as asserted in
the theorem.
\end{rem}

We mention that the same degeneration we use can also be used to give
a simple characteristic-independent proof of the Gieseker-Petri theorem,
and without the need to use the limit linear series construction of
\cite{os8}. In a similar vein, while Lazarfeld's proof \cite{la3} of
the Gieseker-Petri theorem requires characteristic $0$, it seems likely
that his argument yields the classical Brill-Noether theorem in arbitrary
characteristic. Nonetheless, our proof of Theorem \ref{thm:bn-plus-2} has
several desirable characteristics: it gives an all-in-one approach to
both the dimensional statements and nonemptiness statements of the
Brill-Noether theorem; it avoids the additional machinery of the Petri
map; and, although restrictive, the generalization to ramification at
two points ought to be of independent interest. For instance, we can
conclude that spaces of Eisenbud-Harris limit linear series on chains
of curves made up of general marked curves have the expected dimension
in any characteristic; this is not true for arbitrary curves of compact
type.

\subsection*{Acknowledgements} I would like to thank Montserrat Teixidor
i Bigas and Gavril Farkas for helpful conversations.

\section{The base case}

Assuming the machinery of limit linear series, the most technical part
of the proof of Theorem \ref{thm:bn-plus-2} is in fact the base case,
which proceeds by direct case-by-case analysis as follows. 

\begin{lem}\label{lem:base-cases} Theorem \ref{thm:bn-plus-2} holds when
$g=0$ or $1$, and moreover in this case the generality hypothesis can be
described explicitly: for $g=0$, no additional conditions are required,
while for $g=1$, we may take any $C$, and it suffices to assume that $P_1$ 
does not differ from $P_2$ by $m$-torsion, with $m \leq d$.
\end{lem}

\begin{proof}
It will be convenient to set the following notation for 
vanishing sequences: $a^i_j:=\alpha^i_j+j$ for all $i,j$. We begin with 
the following observation, which follows from a dimension count: if 
$(\sL,V)$ is a 
$\fg^r_d$ with vanishing sequence (at least) $a^i_j$ at $P_i$ for $i=1,2$,
then for any $j=0,\dots,r$, then there exists $s \in V$ vanishing to order
at least $a^1_j$ at $P_1$ and $a^2_{r-j}$ at $P_2$. We then easily conclude
the asserted emptiness statement under the generality hypothesis as 
follows: we immediately see that we
must have $a^1_j+a^2_{r-j} \leq d$ for all $j$, which is precisely the
desired statement for $g=0$. On the other hand, if $a^1_j=a$ and 
$a^2_{r-j}=d-a$, we also conclude that we can only have a $\fg^r_d$ with
this vanishing if the underlying line bundle is $\sO(aP_1+(d-a)P_2)$.
This can occur for two different values of $a$ only if 
$\sO(aP_1+(d-a)P_2) \cong \sO(a'P_1+(d-a')P_2)$, or 
$\sO((a-a')P_1+(a'-a)P_2) \cong \sO$. This implies that $P_1-P_2$ is
an $|a-a'|$-torsion point in the Jacobian of $C$, so this in turn
violates our generality hypothesis for the $g=1$ case, and we conclude
the asserted emptiness.

We next consider the dimension assertion, recalling that as discussed
in Remark \ref{rem:reduce-to-special}, the dimensional lower bound is
an immediate consequence of the construction of 
$G^r_d(C,P_1,P_2,\alpha^1,\alpha^2)$. We thus direct our attention to
the upper bound, starting with the case $g=0$.
The only line bundle of degree $d$ is $\sO(d)$, so
the space of all $\fg^r_d$s is the Grassmannian $G(r+1,H^0(\PP^1,\sO(d)))$, 
where we can identify $H^0(\PP^1,\sO(d))$ with the space of polynomials of
degree $d$. Directly from the definitions, we have that 
$G^r_d(C,P_1,P_2,\alpha^1,\alpha^2)$ is described as an intersection of 
two Schubert cycles, associated to the flags of polynomials vanishing
to different orders at $P_1$ and $P_2$, respectively. By unique factorization
of polynomials, these flags are complementary in $H^0(\PP^1,\sO(d))$, so
they are general, and thus by Kleiman's theorem, if the intersection of the 
Schubert cycles is nonempty, it has the expected dimension. 

For the case $g=1$, to prove the dimension statement we only need an
upper bound, and we work over one point of $\Pic^d(C)$ at a time.
According to the Riemann-Roch theorem, there are only two possibilities for 
the vanishing sequence of a complete linear series of a line bundle $\sL$ 
at a point $P$: if $\sL \not \cong \sO(dP)$, the vanishing sequence is 
$0,1,\dots,d-1$, while if $\sL \cong \sO(dP)$, the vanishing sequence is 
$0,1,\dots,d-2,d$.
In particular, if $a^1_r=d$, then $G^r_d(C,P_1,P_2,\alpha^1,\alpha^2)$ is 
necessarily supported over the point corresponding to $\sO(dP_1)$, and
similarly for $a^2_r=d$. By the generality hypothesis, 
$\sO(dP_1) \not \equiv \sO(dP_2)$, so we conclude that in order for the
space to be nonempty, we must have at most one $a^i_r=d$. 

There are two types of line bundles to consider: first, if 
$\sL \not\cong \sO(aP_1+(d-a)P_2)$ for $0 \leq a \leq d$, we see that the
flags obtained by imposing vanishing at $P_1$ and $P_2$ are transverse,
hence general as before, so we conclude that for such an $\sL$, if the
corresponding fiber of $G^r_d(C,P_1,P_2,\alpha^1,\alpha^2)$ is nonempty,
it has dimension 
$$(r+1)(d-1-r)-\sum_{i=1}^2 \sum_{j=0^r}\alpha^i_j=\rho-1.$$ 
Since the space of such line bundles $\sL$ has dimension $1$, we conclude
that $G^r_d(C,P_1,P_2,\alpha^1,\alpha^2)$ can have dimension at most $\rho$
over the open locus of line bundles of this type. 

Next, suppose that
$\sL \cong \sO(aP_1+(d-a)P_2)$ for some $a$. Observe that by our generality
hypothesis, $a$ is always uniquely determined. We first consider the case
that $a=0$ or $d$. Here, we find that the flags
obtained in $H^0(C,\sL)$ from vanishing at $P_1$ or $P_2$ are still
transverse, and thus the Schubert cycles intersect transversely, but the 
Schubert cycle indexing can be off by one from the vanishing sequence if 
$a=0$ and $a^2_r=d$, or if $a=d$ and $a^1_r=d$. In either case, we have
that the dimension of the corresponding fiber is at most
$$(r+1)(d-1-r)-\sum_{i=1}^2 \sum_{j=0^r}\alpha^i_j+1=\rho.$$ 
Finally, if $0<a<d$, we find that the flag indexing corresponds to the
Schubert cycle indexing, but the flags fail to be transverse at the
spaces corresponding to vanishing of order $a$ at $P_1$ and $d-a$ at $P_2$.
This only affects the dimension of intersection if for some $j$, we have 
$a^1_j=a$ and $a^2_{r-j}=d-a$.
Since we have already proved the emptiness assertion if $a^1_j+a^2_{r-j}=d$
for more than one index $j$, we may assume that there is no $j' \neq j$
with $a^1_j+a^2_{r-j}=d$, and in particular if $j>0$, we must have
$a^1_{j-1}\leq a-2$. Thus, regardless of whether or not $j>0$, we can 
still obtain a valid vanishing sequence
if we replace $a^1_j$ by $a-1$. In this case, we have already seen that
the Schubert cycles still intersect in the expected dimension, which is
now 
$$(r+1)(d-1-r)-\sum_{i=1}^2 \sum_{j=0^r}\alpha^i_j+1=\rho$$ 
because we changed $a^1_j$. This gives us the desired statement.

It remains to check the asserted nonemptiness statement, which does not
require any generality hypothesis. In the case
$g=0$, if $a^1_j+a^2_{r-j} \leq d$ for all $r$, we can simply take
$r$ sections of $\sO(d)$, with the $j$th section vanishing to order
exactly $a^1_j$ at $P_1$ and $a^2_{r-j}$ at $P_2$, and we obtain the
desired $\fg^r_d$. Similarly, if $g=1$ and we have $a_1,a_2$ with
$a_1+a_2 \leq d-1$, then for any line bundle $\sL$ on $C$ we have a
section vanishing to order exactly $a_1$ at $P_1$ and $a_2$ at $P_2$ as
long as we do not have $a_1+a_2=d-1$ and $\sL\cong (a_1P_1+(d-a_1)P_2)$
or $\sL \cong \sO((d-a_2)P_1+a_2 P_2)$.
Thus, if $a^1_j+a^2_{r-j} \leq d-1$ for
all $j$, we can choose any $\sL \not \cong \sO(aP_1+(d-a)P_2)$ for 
$a=0,\dots,d$, and then construct the desired $\fg^r_d$ via a basis
as above. Finally, if $a^1_j+a^2_{r-j}=d$ for some $j$,
and $a^1_{j'}+a^2_{r-j} \leq d-1$ for $j' \neq j$, set 
$\sL=\sO(a^1_j P_1+a^2_{r-j}P_2)$, and since we cannot have
$a^1_{j'}=a^1_j-1$ or $a^2_{j'}=a^2_{r-j}-1$ for any $j'$, we are
still able to find a basis of sections with precisely the desired
vanishing. This completes the proof of the lemma.
\end{proof}

\section{The degeneration}

We begin by reviewing the machinery of limit linear series.
For our argument, it is enough to consider degenerations to curves with
two components. Specifically, we will consider the following situation:

\begin{sit} Let $X_0$ be a proper nodal curve over $k$ obtained by gluing
two smooth curves $Y$ and $Z$ to one another at a single node $Q$.
\end{sit}

We recall the Eisenbud-Harris definition of a limit linear series from
\cite{e-h1}:

\begin{defn}\label{def:eh} An {\bf Eisenbud-Harris limit} $\fg^r_d$ on 
$X_0$ is a pair 
$((\sL_Y,V_Y),(\sL_Z,V_Z))$ of $\fg^r_d$'s on $Y$ and $Z$ respectively,
such that for $j=0,\dots,r$ we have
\begin{equation}\label{eq:eh-ineq} \alpha^Y_j+\alpha^Z_{r-j} \geq d-r,
\end{equation}
where $\alpha^Y_0,\dots,\alpha^Y_r$ and $\alpha^Z_0,\dots,\alpha^Z_r$ are 
the ramification sequences at $Q$ of $(\sL_Y,V_Y)$ and $(\sL_Z,V_Z)$, 
respectively.
\end{defn}

\begin{notn}
The space of Eisenbud-Harris limit $\fg^r_d$'s on $X_0$ is denoted by 
$G^{r,\EH}_d(X_0)$.
\end{notn}

\begin{rem} In fact, the Eisenbud-Harris terminology differs slightly in
that our ``limit series'' are their ``crude limit series.'' 
\end{rem}

We also note that we can allow for imposed ramification of limit 
series as follows: given a smooth point $P$ of $X_0$, we say an
Eisenbud-Harris limit $\fg^r_d$ given by the pair $((\sL_Y,V_Y),(\sL_Z,V_Z))$
has ramification at least $\alpha$ at $P$ if $(\sL_Y,V_Y)$ or $(\sL_Z,V_Z)$
does, depending on whether $P$ lies on $Y$ or $Z$. 

\begin{notn}\label{notn:strata} Given nondecreasing integer sequences 
$\alpha^Y,\alpha^Z$ of length $r+1$ in $[0,d-r]$ satisfying
\begin{equation}\label{eq:eh-ineq-2} \alpha^Y_j+\alpha^Z_{r-j} \geq d-r
\end{equation}
for $j=0,\dots,r$, 
let $G^{r,\EH}_d(X_0;\alpha^Y,\alpha^Z)$ be the 
space of Eisenbud-Harris limit $\fg^r_d$'s $((\sL_Y,V_Y),(\sL_Z,V_Z))$ on 
$X_0$ with $(\sL_Y,V_Y)$ having ramification sequence $\alpha^Y$ at $Q$, and
$(\sL_Z,V_Z)$ having ramification sequence $\alpha^Z$ at $Q$.
\end{notn}

The power of the
Eisenbud-Harris theory derives from its inductive structure. Specifically,
we have the following basic observations, which are immediate consequences
of the definition.

\begin{prop}\label{prop:eh-describe} 
The spaces $G^{r,\EH}_d(X_0;\alpha^Y,\alpha^Z)$ give a stratification of
$G^{r,\EH}_d(X_0)$ by (locally closed) subschemes. Moreover, each space 
$G^{r,\EH}_d(X_0;\alpha^Y,\alpha^Z)$ is isomorphic to an open subscheme of
$G^r_d(Y,Q,\alpha^Y) \times G^r_d(Z,Q,\alpha^Z)$.

Alternatively, $G^{r,\EH}_d(X_0)$ is the union over all sequences
$\alpha^Y,\alpha^Z$ achieving equality in \eqref{eq:eh-ineq-2} for all $j$ of
the spaces 
$$G^r_d(Y,Q,\alpha^Y) \times G^r_d(Z,Q,\alpha^Z).$$

The same descriptions hold if ramification is imposed at smooth points of
$X_0$.
\end{prop}

In \cite{os8}, a different definition of limit linear series is given. We
will not need the definition itself, but will need notation for the
resulting space:

\begin{notn} The space of limit $\fg^r_d$'s on $X_0$ as defined in \cite{os8}
is denoted by $G^r_d(X_0)$.
\end{notn}

The two spaces of limit linear series are different, but we have the 
following result comparing them:

\begin{thm}\label{thm:grd-compare} There is a map
\begin{equation}\label{eq:grd-compare}
G^r_d(X_0) \to G^{r,\EH}_d(X_0)
\end{equation}
which is surjective, 
and has fiber dimension bounded as follows:
given an Eisenbud-Harris limit $\fg^r_d$ on $X_0$ as in Definition
\ref{def:eh}, the corresponding fiber of \eqref{eq:grd-compare} has
dimension at most
$$\sum_{j=0}^r \alpha^Y_j+\alpha^Z_{r-j} - d+r,$$
\end{thm}

This result combines Proposition 6.6 
of \cite{os8}
with Corollary 5.5 of \cite{os13}; see also Theorem 2.3 of \cite{li2}
for a simplified proof of the latter.

We can then define 
ramification conditions on $G^r_d(X_0)$ by taking the preimage under
\eqref{eq:grd-compare} of the corresponding locus of $G^{r,\EH}_d(X_0)$.

\begin{rem} Technically, the map \eqref{eq:grd-compare} is not known to
exist in general on a scheme-theoretic level. The precise statement is
that $G^{r,\EH}_d(X_0)$ is a closed subscheme of $G^r_d(Y)\times G^r_d(Z)$,
and \eqref{eq:grd-compare} is in fact a morphism 
$G^r_d(X_0) \to G^r_d(Y)\times G^r_d(Z)$ which factors set-theoretically
through $G^{r,\EH}_d(X_0)$. However, since we are only interested in
questions of (non)emptiness and dimension, this technical distinction is
irrelevant to us.
\end{rem}

We next recall the behavior of limit series in smoothing families.

\begin{sit}\label{sit:smoothing-family} Suppose $X/B$ is a flat, proper 
morphism, with $B$ the
spectrum of a DVR, and $X$ a regular scheme having special fiber $X_0$
as above, and generic fiber $X_{\eta}$ a smooth proper curve.
\end{sit}

\begin{thm}\label{thm:grd-families} There exists a scheme $G^r_d(X/B)$,
proper over $B$, with special fiber $G^r_d(X_0)$, generic fiber 
$G^r_d(X_{\eta})$, and such that every component of $G^r_d(X/B)$ has 
dimension at least $\rho+\dim B=\rho+1$. The same holds with imposed
ramification along smooth sections of $X/B$.
\end{thm}

This is Theorem 5.3 of \cite{os8}.

Using the dimension and properness assertions of the theorem, we conclude:

\begin{cor}\label{cor:grd-families} If $G^r_d(X_0)$ is empty, so is
$G^r_d(X_{\eta})$. If $G^r_d(X_0)$ is nonempty, every component has dimension
at least $\rho$; if it has pure dimension $\rho$,
then $G^r_d(X_{\eta})$ is also nonempty of pure dimension $\rho$. The same 
holds with imposed ramification along smooth sections of $X/B$.
\end{cor}

We can now easily prove Theorem \ref{thm:bn-plus-2}. Rather than explicitly
working with a degeneration to a chain of elliptic curves, it is more
convenient to work by induction on genus, but in either case the spirit
is the same.

\begin{proof}[Proof of Theorem \ref{thm:bn-plus-2}] The base cases
are $g=0,1$, and treated in Lemma \ref{lem:base-cases}. Given $g \geq 2$,
suppose Theorem \ref{thm:bn-plus-2} holds in all genera strictly less than
$g$. Given $r,d,\alpha^1,\alpha^2$, let $Y$ be a smooth genus-$1$ curve with 
points $P_1$, $Q$
not differing by $m$-torsion for $m \leq d$, and let $Z$ be a general
$2$-marked smooth curve of genus $g-1$, with marked points $Q$ and $P_2$.
Let $X_0$ be the nodal curve obtained by gluing $Y$ to $Z$ at $Q$. Given
this choice of $X_0$, let $X/B$ be as in Situation \ref{sit:smoothing-family},
and let $P_1,P_2$ be sections of $X/B$ extending the chosen points of $X_0$.
That such a family always exists is well known; see for instance
Theorem 3.4 of \cite{os8}.
\begin{figure}
\centering
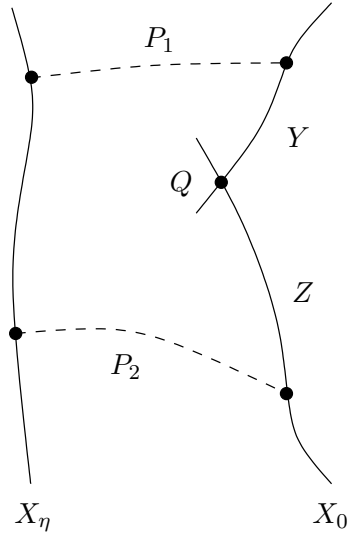
\caption{$Y$ is smooth of genus $1$, and $Z$ is smooth of genus $g-1$.}
\end{figure}

Given sequences $\alpha^Y,\alpha^Z$ as in Notation \ref{notn:strata},
by the induction hypothesis together with Proposition 
\ref{prop:eh-describe}, we conclude that if 
$G^{r,\EH}_{d,a^Y,a^Z}(X_0;(P_1,\alpha^1),(P_2,\alpha^2))$
is nonempty, it has pure dimension 
$$\rho_{Y;\alpha^Y}+\rho_{Z;\alpha^Z}=
\rho-\sum_{j=0}^r (\alpha^Y_j+\alpha^Z_{r-j}-d+r),$$
where 
$$\rho_{Y;\alpha^Y}=
(r+1)(d-r)-r-\sum_{j=0}^r \alpha^1_j-\sum_{j=0}^r \alpha^Y_j$$
and
$$\rho_{Z;\alpha^Z}
=(r+1)(d-r)-(g-1)r-\sum_{j=0}^r \alpha^2_j-\sum_{j=0}^r \alpha^Z_j$$
are the expected dimensions of $\fg^r_d$'s on $Y$ and $Z$ with the
required ramification at $P_1$, $P_2$ and $Q$. Working stratum by stratum,
we thus conclude from Theorem \ref{thm:grd-compare} that 
$G^{r}_d(X_0;(P_1,\alpha^1),(P_2,\alpha^2))$
can have dimension at most $\rho$, and is nonempty if and only if
$G^{r,\EH}_d(X_0;(P_1,\alpha^1),(P_2,\alpha^2))$ is nonempty.
Since we know that 
every component of $G^{r}_d(X_0;(P_1,\alpha^1),(P_2,\alpha^2))$ always
has dimension at least $\rho$, we conclude that
$G^{r}_d(X_0;(P_1,\alpha^1),(P_2,\alpha^2))$ has pure dimension $\rho$
whenever it is nonempty.
We thus conclude from Corollary \ref{cor:grd-families} that if
$G^r_d(X_{\eta};(P_1,\alpha^1),(P_2,\alpha^2))$ is nonempty, it must
have pure dimension $\rho$, and that nonemptiness of
$G^r_d(X_{\eta};(P_1,\alpha^1),(P_2,\alpha^2))$ is equivalent to
nonemptiness of $G^{r,\EH}_d(X_0;(P_1,\alpha^1),(P_2,\alpha^2))$.

According to Remark \ref{rem:reduce-to-special}, we have thus proved the 
dimension portion of the Brill-Noether theorem, and in particular the
emptiness assertion in the case $\rho<0$. It remains to analyze the 
sharp nonemptiness statement. Now, nonemptiness of 
$G^{r,\EH}_d(X_0;(P_1,\alpha^1),(P_2,\alpha^2))$ is equivalent to the
existence of sequences $a^Y,a^Z$ achieving equality in \eqref{eq:eh-ineq-2}
such that the spaces $G^r_d(Y;(P_1,\alpha^1),(Q,\alpha^Y))$ and 
$G^r_d(Z,(Q,\alpha^Z),(P_2,\alpha^2))$ are 
both nonempty. By the induction hypothesis, the latter nonemptiness 
conditions are determined precisely by the inequality \eqref{eq:ram-ineq}.
Consider all possibilities for $\alpha^Y$ such that \eqref{eq:ram-ineq}
is satisfied on $Y$. Then we have $\alpha^1_j+\alpha^Y_{r-j} \leq d-r$ for
all $j$, with equality holding for at most a single value of $j$. Since
$\alpha^Z_j=d-r-\alpha^Y_{r-j}$, we then have
$\alpha^Z_j \geq \alpha^1_j$ for all $j$, with equality holding for at
most a single value of $j$. It follows that as long as \eqref{eq:ram-ineq}
holds on $Y$, then we have
\begin{equation}\label{eq:ram-ineq-2}
\!\!\!\!\!\!\!\!\!\! \sum
_{\,\,\,\,\,\,\,\,\,\,\,\,j:\alpha^Z_j+\alpha^2_{r-j}+r \geq d+1-(g-1)} 
\!\!\!\!\!\!\!\!\!\!\!\!\!\!\! \alpha^Z_j+\alpha^2_{r-j}+r-d+g-1
\geq \left(
\!\!\!\!\!\!\!\!\!\! \sum
_{\,\,\,\,\,\,\,\,\,\,\,\,j:\alpha^1_j+\alpha^2_{r-j}+r \geq d+1-g} 
\!\!\!\!\!\!\!\!\!\!\!\!\!\!\!\alpha^1_j+\alpha^2_j+r-d+g\right)-1.
\end{equation}
It immediately follows that if \eqref{eq:ram-ineq} is violated for 
$\alpha^1,\alpha^2$, then it is impossible to find $\alpha^Y,\alpha^Z$
satisfying \eqref{eq:ram-ineq} on both $Y$ and $Z$. We thus conclude
the desired emptiness statement for 
$G^{r,\EH}_d(X_0;(P_1,\alpha^1),(P_2,\alpha^2))$, and hence for
$G^r_d(X_{\eta};(P_1,\alpha^1),(P_2,\alpha^2))$.

Conversely, suppose first that \eqref{eq:ram-ineq} is satisfied with
equality. It is enough to see that we can choose $\alpha^Y,\alpha^Z$
so that \eqref{eq:ram-ineq-2} is satisfied with equality. Since $g>0$,
we must have at least some indices $j$ such that 
$\alpha^1_j+\alpha^2_{r-j}+r \geq d+1-g$; let $j_0$ be the minimal 
such index. Then we set $\alpha^Y_j=d-r-1-\alpha^1_{r-j}$ for all $j \neq j_0$,
and $\alpha^Y_{j_0}=d-r-\alpha^1_{r-j}$. Note that this gives a valid 
nondecreasing sequence because of the minimal of $j_0$.
If we then set $\alpha^Z_j=d-r-\alpha^Y_{r-j}$ for all $j$, we will
achieve equality in \eqref{eq:ram-ineq-2}, as desired. Lastly, suppose
that \eqref{eq:ram-ineq} is satisfied with strict inequality. Then we
can set  
$\alpha^Y_j=d-r-1-\alpha^1_{r-j}$ and $\alpha^Z_j=d-r-\alpha^Y_{r-j}$ for 
all $j$, and we will still have \eqref{eq:ram-ineq} satisfied on $Z$.
We thus conclude that if \eqref{eq:ram-ineq} is satisfied,
$G^{r,\EH}_d(X_0;(P_1,\alpha^1),(P_2,\alpha^2))$ -- and hence 
$G^r_d(X_{\eta};(P_1,\alpha^1),(P_2,\alpha^2))$ -- is nonempty.
Again according to Remark \ref{rem:reduce-to-special}, we conclude the
sharp nonemptiness assertion of Theorem \ref{thm:bn-plus-2}, completing
the proof of the theorem.
\end{proof}

\begin{rem}\label{rem:technicality} 
The importance of the limit linear series construction of
\cite{os8} arises in proving the dimension upper bound in the case that 
$\rho \geq 0$. In \cite{e-h1}, Eisenbud and Harris do not construct a 
proper family of limit linear
series, instead constructing a space whose special fiber consists of
``refined limit series'' -- those which satisfy \eqref{eq:eh-ineq} with
equality. In characteristic $0$, this does not cause serious problems,
as they show that a limit series arising from a linear series in a 
smoothing family is refined if and only if there is no ramification
specializing to the node. They can then start with any linear series on
the generic fiber and obtain a refined limit series on the special
fiber after base change and blowup.  However, this argument fails in
positive characteristic, again due to the presence of inseparability.
Thus, when working in positive characteristic, there is no obvious way
to reduce to the refined case when comparing dimensions on the special
and generic fibers.
\end{rem}

\bibliographystyle{hamsplain}
\bibliography{hgen}

\end{document}